\documentclass[12pt]{amsart}
\usepackage{amssymb}

\usepackage[latin1]{inputenc}

\usepackage[all]{xy}

\usepackage{cases}
\usepackage{hyperref}

\usepackage{enumerate} \usepackage{float} 
\newcommand{\Z}{{\mathbb Z}} \newcommand{\Q}{{\mathbb Q}}

 \newcommand{\F}{{\mathbb F}}

\newcommand{\Hom}{\operatorname{Hom}\nolimits}

\renewcommand{\ker}{\operatorname{Ker}\nolimits}

\newcommand{\rk}{\operatorname{rk}\nolimits}
\renewcommand{\dim}{\operatorname{dim}\nolimits}

\newcommand{\A}{\ifmmode{\mathcal{A}}\else${\mathcal{A}}$\fi}
\newcommand{\B}{\ifmmode{\mathcal{B}}\else${\mathcal{B}}$\fi}
\newcommand{\C}{\ifmmode{\mathcal{C}}\else${\mathcal{C}}$\fi}
\newcommand{\D}{\ifmmode{\mathcal{D}}\else${\mathcal{D}}$\fi}
\newcommand{\G}{\ifmmode{\mathcal{G}}\else${\mathcal{G}}$\fi}
\newcommand{\I}{\ifmmode{\mathcal{I}}\else${\mathcal{I}}$\fi}
\newcommand{\J}{\ifmmode{\mathcal{J}}\else${\mathcal{J}}$\fi}
\newcommand{\K}{\ifmmode{\mathcal{K}}\else${\mathcal{K}}$\fi}
\renewcommand{\O}{\ifmmode{\mathcal{O}}\else${\mathcal{O}}$\fi}
\renewcommand{\P}{\ifmmode{\mathcal{P}}\else${\mathcal{P}}$\fi}
\newcommand{\U}{\ifmmode{\mathcal{U}}\else${\mathcal{U}}$\fi}
\newcommand{\M}{\ifmmode{\mathcal{M}}\else${\mathcal{M}}$\fi}
\newcommand{\N}{\ifmmode{\mathcal{N}}\else${\mathcal{N}}$\fi}
\newcommand{\Ss}{\ifmmode{\mathcal{S}}\else${\mathcal{S}}$\fi}
\newcommand{\T}{\ifmmode{\mathcal{T}}\else${\mathcal{T}}$\fi}
\newcommand{\Ff}{\ifmmode{\mathcal{F}}\else${\mathcal{F}}$\fi}
\newcommand{\Ll}{\ifmmode{\mathcal{L}}\else${\mathcal{L}}$\fi}
\newtheorem{Thm}{Theorem}[section]
\newtheorem{Conj}{Conjecture}[section]
\newtheorem{Prop}[Thm]{Proposition}

\newtheorem{Lem}[Thm]{Lemma}

\newtheorem*{Maintheorem}{Main Theorem}

\theoremstyle{definition}
\newtheorem{Defi}[Thm]{Definition}

\theoremstyle{remark}
\newtheorem*{Not}{Notation}

\usepackage{color}

\theoremstyle{plain}

\title{Cohomology of $p$-groups of nilpotency class smaller than $p$}

\author{Antonio D\'iaz Ramos}
\address{Departamento de \'Algebra, Geometr\'ia y Topolog\'ia,
Universidad de M\'alaga, Apdo correos 59, 29080 M\'alaga, Spain}
\email{adiazramos@uma.es}

\author{Oihana Garaialde Oca\~{n}a}
\address{Matematika Saila,
Euskal Herriko Unibertsitatearen Zientzia eta Teknologia Fakultatea,
 posta-kutxa 644, 48080 Bilbo, Spain
}
\email{oihana.garayalde@ehu.es}
\author{Jon Gonz\'alez-S\'anchez}
\address{Departamento de Matemáticas, 
Facultad de Ciencia y Tecnología de la Universidad del País Vasco,
 Apdo correos 644, 48080 Bilbao, Spain}
\email{jon.gonzalez@ehu.es}

\date{\today}

\begin{document}

\maketitle

\begin{abstract}
Let $p$ be a prime number, let $d$ be an integer and let $G$ be a $d$-generated finite $p$-group of nilpotency class smaller than $p$. Then the number of possible isomorphism types for the mod $p$ cohomology algebra $H^*(G;\F_p)$ is bounded in terms of $p$ and $d$. 
\end{abstract}

\section{Introduction}
Let $p$ be a prime number and let $G$ be a finite $p$-group of order $p^n$ and nilpotency class $c$, then $G$ has coclass $m=n-c$. In 2005, J. F. Carlson proved that for a fixed integer $m$, there are only finitely many possible isomorphism types for the mod $2$ cohomology algebra of $2$-groups of coclass $m$ (see \cite[Theorem 5.1]{Carlson}). In the same paper, J. F. Carlson conjectures that the analogous result should hold for the $p$ odd case, that is, he conjectures that there are finitely many isomorphism types of cohomology algebras for all $p$-groups of fixed coclass. This result has been partially proved first by B. Eick and D. Green in \cite{BettinaDavid2015} and later by the authors of this manuscript in \cite{DGG}. In the former paper, Eick and Green prove Carlson's conjecture for the reduced mod $p$ algebras, that is, for the quotients $H^*(G;k)/\text{nil}(H^*(G;k))$ where $\text{nil}(H^*(G;k))$ denotes the ideal generated by the nilpotent elements of the algebra $H^*(G;k)$. To that aim, they show that there are finitely Quillen categories of $p$-groups of fixed coclass $m$. In the latter paper, D\'iaz Ramos, Garaialde Ocaña and Gonz\'alez-S\'anchez prove Carlson's conjecture for the \emph{non-twisted} $p$-groups of coclass $m$ \cite[Theorem 7.1]{DGG}. For instance, all the $2$-groups of fixed coclass are non-twisted.

The general case of the above conjecture is still open. One of the main differences between the $p=2$ case and the $p$ odd case is that in the even case, $2$-groups of coclass $m$ have a `large' abelian subsection, while in the $p$ odd case, such $p$-groups have a `large' subsection of nilpotency class $2$.

A dual problem to Carlson's conjecture would be to study the number of  possible isomorphism types for the mod $p$ cohomology algebras $H^*(G;\F_p)$ of $d$-generated finite $p$-groups $G$ of nilpotency class $c$. In this note we prove that whenever the nilpotency class $c$ of a $d$-generated finite $p$-group $G$ is smaller than $p$, then the number of possible isomorphism types of cohomology algebras $H^*(G;\F_p)$ is bounded in terms of $p$ and $d$.

\begin{Maintheorem} 
The number of possible isomorphism types for the mod $p$-cohomology algebra of a $d$-generated finite $p$-group of nilpotency class smaller than $p$ is bounded by a function depending on $p$ and $d$.
\end{Maintheorem}

The key ingredients for proving the above result are the Lazard correspondence, a structure result for finite nilpotent $\Z_p$-Lie algebras, the description of the mod $p$-cohomology of powerful $p$-central $p$-groups with the $\Omega$-extension property given by T. S. Weigel (see \cite{Weigel}) and some refinements on Carlson counting arguments from \cite{Carlson}.

\begin{Not}\label{reducedandnilpotentcohomology}
Let $G$ be a group, $G^{p^k}$ denotes the subgroup generated by the $p^k$ powers of $G$ and $\Omega_k(G)$ denotes the subgroup generated by the elements of order $p^k$. The
 {\it reduced} mod $p$ cohomology ring $H^*(G;\F_p)_{\text{red}}$ is the quotient $H^*(G;\F_p)/nil(H^*(G;\F_p))$, where $nil(H^*(G;\F_p))$ is the ideal of all nilpotent elements in the mod $p$ cohomology. We will use $[,]$ to denote both the Lie bracket of a Lie algebra and the group commutator.  If $D$ is either a group or a Lie algebra and $I$ is either a normal subgroup or an ideal of $D$ repectively, then we denote
$$
[I,{}_{c}D]=[I,\overbrace{D, \cdots, D}^{c}].
$$
We say that a function $k=k(a,b,\ldots,A,B,\ldots )$ is $(a,b,\ldots)$-bounded if there exist a function $f$ depending on $(a,b,\ldots )$ such that $k\leq f(a,b\ldots )$. The rank of a $p$-group $G$ is the sectional rank, that is, 
\[
\rk(G)=\text{max} \{d(H)\; |\; \text{for all} \; H\leq G\},
\]
where $d(H)$ is the minimal number of generators of $H$.
\end{Not}

\section{Preliminaries}

\subsection{Powerful $p$-central groups and $\Z_p$-Lie algebras.}
Following \cite{Weigel}, we define powerful $p$-central groups and $\Z_p$-Lie algebras with the $\Omega$-extension property ($\Omega$EP for short). We also recall a result about the cohomology algebra of such $p$-groups. Recall that for a $p$-group $G$ we denote
$$
\Omega_1(G)=\langle g\in G\mid g^p=1\rangle,
$$
and for a $\Z_p$-Lie algebra $L$ we denote 
$$
\Omega_1(L)=\{a\in L\mid p.a=0\}.
$$


\begin{Defi}\label{defi: powerfulpcentralomega} Let $p$ be an odd prime. Let $G$ be a $p$-group and let $L$ be a finite $\Z_p$-Lie algebra. Then, 
\begin{enumerate}
\item $G$ is \emph{powerful} if $[G,G] \subset G^p$ and $L$ is \emph{powerful} if $[L,L] \subset pL$.
\item $G$ is \emph{$p$-central} if its elements of order $p$ are contained in the center of $G$ and $L$ is \emph{$p$-central} if its elements of order $p$ are contained in the center of $L$.
\item $G$ has the \emph{$\Omega$-extension property} ($\Omega$EP for short) if there exists a $p$-central group $H$ such that $G= H/ \Omega_1(H)$ and $L$ has the \emph{$\Omega$-extension property} if there exists a $p$-central Lie algebra $A$ such that $L= A/ \Omega_1(A)$.

\end{enumerate}
\end{Defi}

An easy computation shows that if $G$ (or $L$) has the $\Omega$EP, then $G$ (or $L$) is $p$-central. We finish this subsection by describing the 
mod $p$ cohomology algebra of a powerful $p$-central $p$-group with $\Omega$EP.


\begin{Thm}\label{Thm:Weigel}Let $p$ be an odd prime, let $G$ be a powerful $p$-central $p$-group with the $\Omega$EP and let $d$ denote the $\F_p$-rank of $\Omega_1(G)$. Then,
\begin{itemize}
\item[(a)] $H^*(G;\F_p)\cong \Lambda(y_1, \dots, y_d) \otimes \F_p[x_1, \dots, x_d]$ with $|y_i|=1$ and $|x_i|=2$,
\item[(b)] the reduced restriction map $j_{\text{red}}: H^*(G;\F_p)_{\text{red}} \to H^*(\Omega_1(G);\F_p)_{\text{red}}$ is an isomorphism.
\end{itemize}
\end{Thm}

\begin{proof}
See \cite[Theorem 2.1 and Corollary 4.2]{Weigel}
\end{proof}

\subsection{The Lazard correspondence for finitely generated nilpotent pro-$p$ groups and $\Z_p$-Lie algebras}


Let $Lie_p$ and $Gr_p$ denote the category of finitelly generated $\Z_p$-Lie algebras of nilpotency class smaller than $p$ and the category of finitelly generated pro-$p$ groups of nilpotency class smaller than $p$, respectively. Denote by $Top$ the category of topological spaces and by $\textbf{for}: Lie_p\to Top$ and  $\textbf{For}: Gr_p\to Top$ the fully faithful forgetful functors which forget the algebraic structure.


\begin{Thm}[\cite{La} Lazard Correspondence]\label{thm: Lazard}
There exist isomorphisms of categories one inverse of each other 
\begin{eqnarray*}
\textbf{exp} : Lie_p\longrightarrow Gr_p \; \; \text{and} \; \; 
\textbf{log} :Gr_p\longrightarrow Lie_p,
\end{eqnarray*}
such that
\begin{equation}
\textbf{for}\circ \textbf{log} =\textbf{For} \text{ \ \ and \ \ } \textbf{For}\circ \textbf{exp}=\textbf{for}.
\end{equation}
Furthermore, if $G\in Gr_p$ and $K$ is a subgroup of $G$,  the following statements hold:
\begin{itemize}
\item[(a)] $K$ is a normal subgroup of $G$ if and only if $\textbf{log}(K)$ is an ideal in $\textbf{log}(G)$. Moreover,  $\textbf{log}(G/K)=\textbf{log}(G)/\textbf{log}(K)$,
\item[(b)] $\textbf{log}(\Omega_1(G))=\Omega_1(\textbf{log}(G))$ and $\textbf{log}(G^p)=p\textbf{log}(G)$,
\item [(c)] Nilpotency class of $G$ = nilpotency class of $\textbf{log}(G)$,
\item [(d)] $G$ is a powerful $p$-group if and only if $\textbf{log}(G)$ is a powerful $\Z_p$-Lie algebra,
\item [(e)] $G$ is a $p$-central group if and only if $\textbf{log}(G)$ is a $p$-central $\Z_p$-Lie algebra,
\item [(f)] A subset $X$ of $K$ generates $K$ (topologically) if and only if $X$ generates $\textbf{log} (K)$ as a $\Z_p$-Lie algebra. In particular the number of generators of $K$ as a topological group coincides with the number of generators of $\textbf{log} (K)$ as a $\Z_p$-Lie algebra.
\end{itemize}
\end{Thm}

\begin{proof}
The existence of the functors $\textbf{exp}$ and $\textbf{log}$ was first discovered by M. Lazard in \cite{La}. Explicit formulae can be found in \cite{CicaloGraafLee}.  One can also found 
the Lazard correspondence in \cite[Theorem 10.13 and page 124]{khukhro}. More precisely: (a) follows from \cite[Theorem 10.13 (b)]{khukhro}, (b) follows from the explicit Baker-Campbell-Hausdorff formula \cite[Lemma 9.15]{khukhro}, (c) follows from  \cite[Theorem 10.13 (d)]{khukhro}, (d) and (e) follow from comparing (b) of this theorem and  \cite[Theorem 10.13 (c)]{khukhro}. Finally, (f) follows from the Baker-Campbell-Hausdorff formulae.
\end{proof}

We finish this subsection by proving a result on finite $p$-groups of nilpotency class smaller than $p$.

\begin{Prop}
\label{prop:GpOmega}
Let $p$ be a prime number and let $G$ be a finite $p$-group of nilpotency class smaller than $p$. Then $G^p$ is powerful $p$-central $p$-group with the $\Omega$EP.
\end{Prop}

\begin{proof}
Let $G$ be a finite $p$-group of nilpotency class $m<p$. By the Lazard correspondence, $G=\textbf{exp}(L)$ where $(L,+, [,])$ is a $\Z_p$-Lie algebra of nilpotency class $m$. 

The $\Z_p$-Lie algebra $(pL, +,[,])$ satisfies the following properties
\[
[pL,pL]=p[pL,L]\leq p(pL) \; \text{and} \; [\Omega_1(pL), pL]=p[\Omega_1(pL),L]=0.
\]
That is, $pL$ is a powerful $p$-central $\Z_p$-Lie algebra of nilpotency class at most $m$ (see Definition \ref{defi: powerfulpcentralomega}). By Theorem \ref{thm: Lazard}(b), $G^p=\textbf{exp}(pL)$ is a powerful $p$-central $p$-group of nilpotency class at most $m$.

It remains to show that $G^p$ has the $\Omega$EP. To that aim, consider $(L,+)$ as a $d$-generated abelian $p$-group for some positive integer $d$. Then, $L$ is isomorphic to the quotient of a free $\Z_p$-module $M$ with $d$ generators by a $\Z_p$-submodule $I$ with $I\subseteq pM$. The Lie bracket $[,]\in \Hom(L\otimes_{\Z_p}L,L)$ is an antisymmetric bilinear form and it can be uniquely lifted to the bilinear form $\{,\}\in \Hom(M\otimes_{\Z_p}M,M)$. Although this bilinear form $\{,\}$ may fail to be a Lie bracket in $M$, for all $x,y,z\in M$, we have that
\[
\{x,y,z\}+\{z,x,y\}+\{y,z,x\}\in I
\]
holds, since $[,]$ satisfies the Jacobi identity in $L=M/I$. In particular, the following property
\[
\{px,py,pz\}+\{pz,px,py\}+\{py,pz,px\}\in p^3I,
\]
shows that the bilinear form $\{,\}$ is a Lie bracket in the quotient $pM/pI$. Moreover, the nilpotency class of $(pM/pI, +, \{,\})$ is at most $m$ as
\[
\{\overbrace{pM,\ldots,pM}^{m+1}\}= p^{m+1}\{\overbrace{M,\ldots,M}^{m+1}\}\leq p^{m+1}I.
\]
As before, the properties
\[
\{pM,pM\}\leq pM\; \text{and} \; \{\Omega_1(pM),pM\}=0
\]
show that $(pM/pI, + \{,\})$ is a powerful $p$-central $\Z_p$-Lie algebra. Notice that $\Omega_1(pM/pI)=I/pI$ and by the third isomorphism theorem
$$pL=p(M/I)=pM/I\cong \frac{pM/pI}{I/pI}=\frac{pM/pI}{\Omega_1(pM/pI)}.$$
Finally, by Theorem \ref{thm: Lazard}, we have that $\textbf{exp}(pM/pI)$ is $p$-central and, as
\[
G^p=\textbf{exp}(pL)=\textbf{exp}\left(\frac{pM/pI}{\Omega_1(pM/pI)}\right)=\frac{\textbf{exp}(pM/pI)}{\textbf{exp}(\Omega_1(pM/pI)}=\frac{\textbf{exp}(pM/pI)}{\Omega_1(\textbf{exp}(pM/pI))},
\]
we conclude that $G^p$ is a powerful $p$-central $p$-group with the $\Omega$EP.
\end{proof}

\section{Finitely generated nilpotent $\Z_p$-Lie algebras}\label{sec: liealgebras}

We start this section by giving an easy result that will be used throughout the section.

\begin{Lem}\label{inclusioncentralseries} Let $L$ be a $\Z_p$-Lie algebra and let $I$ be an ideal in $L$. Then, for all $i\geq 1$,
\[
[\gamma_i(L), I]\subseteq [I, {}_{i}L].
\]
where $\gamma_i(L)$ denotes the lower central series of $L$. In particular, for all $i,j\geq 1$, we have that $[\gamma_i(L), \gamma_j(L)]\subseteq \gamma_{i+j}(L)$.
\end{Lem}

\begin{proof} We shall prove the result by induction on $i$. If $i=1$, then the statement clearly holds. Suppose that the above inclusion holds for all $i< k$. Then,
\begin{align*}
[\gamma_k(L),I]&=[[\gamma_{k-1}(L),L],I]\subseteq [[I,\gamma_{k-1}(L)],L]+[[L,I],\gamma_{k-1}(L)]\\
&\subseteq [[I, {}_{k-1}L],L]+[[L,I],{}_{k-1}L]= [I,{}_{k}L],
\end{align*}
where the first inclusion holds by the Jacobi identity and the second one uses the induction hypothesis.
\end{proof}

Let $p$ be a prime number, let $c,d$ be positive integers, let $X$ be a set with $d$ elements and let $L_c(X)$ denote the free $\Z_p$-Lie algebra over $X$ of nilpotency class $c$. Let $\widetilde{L}_c(X)= L_c(X)\otimes_{\Z_p} \Q_p$ denote the free $\Q_p$-Lie algebra of nilpotency class $c$. We can assume that $L_c(X)\subseteq \tilde{L}_c(X)$. We define a $\Z_p$-Lie algebra,
\[
\widehat{L}_c(X):= L_c(X)+\frac{1}{p}\gamma_2(L_c(X))+ \frac{1}{p^2}\gamma_3(L_c(X))+ \dots +\frac{1}{p^{c-1}}\gamma_{c}(L_c(X))\subseteq \widetilde{L}_c(X).
\]

\begin{Lem}\label{Lem: Lpowerful}
Let $\widehat{L}_c(X)$ be as defined above. Then
\begin{enumerate}
\item[(a)] $\widehat{L}_c(X)$ is a finitely generated powerful $\Z_p$-Lie algebra of nilpotency class $c$.
\item[(b)] The rank of $\widehat{L}_c(X)$ is $(c,d)$-bounded.
\item[(c)] The index $|\widehat{L}_c(X):L_c(X)|$ is $(p,c,d)$-bounded.
\end{enumerate}
\end{Lem}

\begin{proof} We start by showing that the nilpotency class of $\widehat{L}_c(X)$ is $c$. Since $\widehat{L}_c(X)$ is a $\Z_p$-Lie subalgebra of  $\widetilde{L}_c(X)$, the nilpotency class of $\widehat{L}_c(X)$ is at most $c$. Since  $\widehat{L}_c(X)$ contains $L_c(X)$ the nilpotency class of  $\widehat{L}_c(X)$ is at least $c$. This shows that the nilpotency class of  $\widehat{L}_c(X)$ is precisely $c$. 

Let us check that the $\Z_p$-Lie algebra $\widehat{L}_c(X)$ is powerful. By the linearity of the Lie bracket $[,]$ in $\widetilde{L}_c(X)$, it suffices to show that for every $i,j\geq 1$, the following equality holds:
$$
[\frac{1}{p^{i-1}}\gamma_i(L_c(X)), \frac{1}{p^{j-1}}\gamma_j(L_c(X))]\subset p\widehat{L}_c(X).
$$
Indeed, for every $i,j\geq 1$, we have that
\begin{align*}
[\frac{1}{p^{i-1}}\gamma_i(L_c(X)), \frac{1}{p^{j-1}}\gamma_j(L_c(X))]&=\frac{1}{p^{i+j-2}}[\gamma_i(L_c(X)), \gamma_j(L_c(X))]\\
&\subseteq p(\frac{1}{p^{i+j-1}}\gamma_{i+j}(L_c(X))) \subset p \widehat{L}_c(X).
\end{align*}
So, $\widehat{L}_c(X)$ is powerful.

Next, we show that the rank of $\widehat{L}_c(X)$ is bounded by a function depending on $c$ and $d$. Consider $\widetilde{L}_c(X)=L_c(X)\underset{\Z_p}\otimes \Q_p$ as a $\Q_p$-vector space and decompose it as
\[
\widetilde{L}_c(X)=\text{Span}(X)\oplus \text{Span}([X,X])\oplus \cdots \oplus \text{Span}([X, \overset{c}\dots, X]),
\]
where the dimension of $\text{Span}([X, \overset{i}\dots, X])$ as $\Q_p$-vector space is smaller than equal to $d^i$. Then, we have that the dimension of $\widetilde{L}_c(X)$ as a $\Q_p$-vector space is
\[
\dim_{\Q_p}(\widetilde{L}_c(X))\leq d+d^2+\dots +d^c=\frac{d^{c+1}-d}{d-1}.
\]
Notice that  $\widehat{L}_c(X)$ has no torsion and recall that $\widehat{L}_c(X)\otimes_{\Z_p} \Q_p=L_c(X)\otimes_{\Z_p} \Q_p=\widetilde{L}_c(X)$. Thus 
$$
\rk(\widehat{L}_c(X))=\dim_{\Q_p}(\widetilde{L}_c(X))\leq \frac{d^{c+1}-d}{d-1}
$$
proves item (b).

Finally, since $p^{c-1} \widehat{L}_c(X)\subseteq L_c(X)$, we have that
\[
|\widehat{L}_c(X):L_c(X)|\leq p^{(c-1)r},
\]
where $r=\rk(G)$.
\end{proof}
Let $I$ be an ideal of $L_c(X)$ such that $|L_c(X):I|<\infty$. This ideal is not necessarily an ideal of $\widehat{L}_c(X)$ and thus, we extend $I$ to an ideal $\widehat{I}$ of $\widehat{L}_c(X)$ as follows:

\[
\widehat{I}:= I+ \frac{1}{p}[I,L_c(X)]+ \frac{1}{p^2} [[I,L_c(X)],L_c(X)]+ \dots + \frac{1}{p^{c-1}}[I, _{(c-1)}L_c(X)].
\]

\begin{Lem}Let $\widehat{L}_c(X)$ and $\widehat{I}$ be as above. Then 
$$[\widehat{I}, \widehat{L}_c(X)]\subset p\widehat{I},$$
and, in particular, $\widehat{I}$ is an ideal in $\widehat{L}_c(X)$.
\end{Lem}

\begin{proof}
For all $i\geq 1$ and $j\geq 1$, we have that
\begin{align*}
[\frac{1}{p^i} [I,_{i}L_c(X)], \frac{1}{p^{j-1}}\gamma_j(L_c(X))]&=\frac{1}{p^{i+j-1}}[[I, {}_{i}L_c(X)],\gamma_j(L_c(X))]\\
&\subset \frac{1}{p^{i+j-1}}[[I, {}_{i}L_c(X)],{}_{j}L_c(X)]\\
&=\frac{1}{p^{i+j-1}}[I, {}_{i+j} L_c(X)]= p(\frac{1}{p^{i+j}}[I, {}_{i+j}L_c(X)])\\
&\subset p\widehat{I},
\end{align*}
where in the first inclusion we used Lemma \ref{inclusioncentralseries}. By the linearity property of the Lie bracket, we conclude that $[\widehat{I}, \widehat{L}_c(X)]\subset p\widehat{I}$.
\end{proof}

\begin{Prop}\label{prop: powerfulpcentralrank} The quotient $\widehat{L}_c(X)/\widehat{I}$ is a finite powerful $p$-central $\Z_p$-Lie algebra with the $\Omega$EP and with $(c,d)$-bounded rank.
\end{Prop}

\begin{proof} We start by showing that the quotient $\widehat{L}_c(X)/\widehat{I}$ is powerful. From Lemma \ref{Lem: Lpowerful} (a), we have that the $\Z_p$-Lie algebra $\widehat{L}_c(X)$ is powerful and since this property is inherited to factor Lie algebras, our claim holds. Also, the rank of $\widehat{L}_c(X)/\widehat{I}$ is at most the rank of $\widehat{L}_c(X)$ which is $(c,d)$-bounded by Lemma \ref{Lem: Lpowerful} (b).

Now, consider $p\widehat{I} \subset \widehat{I}$ which is an ideal of $\widehat{L}_c(X)$. It is straightforward to see that 
$\Omega_1(\widehat{L}_c(X)/p\widehat{I})=\widehat{I}/p\widehat{I}$ and therefore, 
\begin{align*}
\frac{\widehat{L}_c(X)/p\widehat{I}}{\Omega_1(\widehat{L}_c(X)/p\widehat{I})}=\frac{\widehat{L}_c(X)/p\widehat{I}}{\widehat{I}/p\widehat{I}}\cong \widehat{L}_c(X)/\widehat{I}.
\end{align*}
Thus, to prove that $\widehat{L}_c(X)/\widehat{I}$ has $\Omega$EP it is enough to 
 show that $\widehat{L}_c(X)/p\widehat{I}$ is a $p$-central $\Z_p$-Lie algebra. Indeed, by the previous lemma $[\widehat{I}, \widehat{L}_c(X)]\subset p\widehat{I},$ which shows that
\begin{align*}
[\Omega_1(\widehat{L}_c(X)/p\widehat{I}), \widehat{L}_c(X)/p\widehat{I}]&=[\widehat{I}/p\widehat{I}, \widehat{L}_c(X)/p\widehat{I}]=0.
\end{align*} 
Hence, we conclude that $\widehat{L}_c(X)/\widehat{I}$ is a powerful $p$-central $\Z_p$-Lie algebra with the $\Omega$EP.
\end{proof}



Let $I$ be a Lie ideal of $L_c(X)$ and let $\varphi: L_c(X)/I \to \widehat{L}_c(X)/\widehat{I}$ be the natural map that sends an element $a+I$ to $a+\widehat{I}$. Then
\[
\ker \varphi\cong \frac{L_c(X)\cap\widehat{I}}{I} \; \; \text{and} \; \; \varphi\left(\frac{L_c(X)}{I}\right)\cong \frac{L_c(X)}{L_c(X)\cap\widehat{I}}\cong \frac{L_c(X)+\widehat{I}}{\widehat{I}}.
\]
Thus, for any ideal $I$ there is an extension of $\Z_p$-Lie algebras
\begin{equation}\label{extensionofLiealgebras}
0 \to\frac{L_c(X)\cap \widehat{I}}{I} \hookrightarrow \frac{L_c(X)}{I} \twoheadrightarrow \frac{L_c(X)+\widehat{I}}{\widehat{I}}\to 0,
\end{equation}
satisfying the properties of the following proposition.

\begin{Prop} \label{Prop:bounds}Let $d,c$ be positive integers, let $X$ be a set with $d$ elements and let $I$ be an ideal of $L_c(X)$. Under the above notation, the following properties hold:
\begin{itemize}
\item[(a)] The order $|\frac{L_c(X)\cap\widehat{I}}{I}|$ is $(p,c,d)$-bounded.
\item[(b)] The index $|\tfrac{\widehat{L}_c(X)}{\widehat{I}}:\tfrac{L_c(X)+\widehat{I}}{\widehat{I}}|$ is $(p,c,d)$-bounded.
\end{itemize}
\end{Prop}

\begin{proof}
We start by proving that the order $|\frac{L_c(X)\cap\widehat{I}}{I}|$ is $(p,c,d)$-bounded. To that aim, we shall show that both its rank and its exponent are $(p,c,d)$-bounded. Let $r$ denote the rank of $\widehat{L}_c(X)$ and note that $\rk(\frac{L_c(X)\cap\widehat{I}}{I})\leq r$. By Lemma \ref{Lem: Lpowerful} (b), $r$ is $(c,d)$-bounded and thus, $\rk(\frac{L_c(X)\cap\widehat{I}}{I})$ is also $(c,d)$-bounded. Moreover, as $L_c(X)\cap \widehat{I}\subseteq \widehat{I}$ and $p^{c-1}\widehat{I}\subseteq I$, the exponent of $\frac{L_c(X)\cap\widehat{I}}{I}$ is also $(p,c)$-bounded. Thus, the first claim holds.

The last claim follows from the fact that 
\[
\left|\frac{\widehat{L}_c(X)}{\widehat{I}}:\frac{L_c(X)+\widehat{I}}{\widehat{I}}\right|=|\widehat{L}_c(X):L_c(X)+\widehat{I}|\leq |\widehat{L}_c(X):L_c(X)|,
\]
and from Lemma \ref{Lem: Lpowerful} (c).
\end{proof}

\begin{Thm} \label{theorem:Liealgebras_structure}
Let $p$ be a prime number and let $L$ be a $d$-generated finite $\Z_p$-Lie algebra of nilpotency class $c$. Then there exist a powerful $p$-central $\Z_p$-Lie algebra $\widehat{L}$ with $(d,c)$-bounded number of generators and an ideal $J$ of $L$ such that
\begin{enumerate}
\item[(a)] $|J|$ is $(p,c,d)$-bounded.
\item[(b)] $L/J$ can be embedded as a subalgebra in $\widehat{L}/\Omega_1(\widehat{L})$.
\item[(c)] $|\widehat{L}/\Omega_1(\widehat{L}):L/J|$ is $(p,c,d)$-bounded.
\end{enumerate}
\end{Thm}

\begin{proof}Let $X$ denote a generating a set of the $d$-generated finite $\Z_p$-Lie algebra $L$ of nilpotency class $c$. Let $L_c(X)$ denote the free $\Z_p$-Lie algebra of nilpotency class $c$ and let $\pi: L_c(X) \twoheadrightarrow L$ be the projection map. Consider $I=\ker \pi$ so that $L\cong \frac{L_c(X)}{I}$. Then, $L$ fits into an extension of the form \eqref{extensionofLiealgebras}.

By abusing the notation, take
\[
J=\frac{L_c(X)\cap \widehat{I}}{I}\subseteq L_c(X)/I\cong L,
\]
of $(p,c,d)$-bounded order (see Proposition \ref{Prop:bounds} (a)). Now, 
\[
L/J\cong \frac{L_c(X)}{L_c(X)\cap \widehat{I}}\cong \frac{L_c(X)+\widehat{I}}{\widehat{I}}
\]
can be embedded in $\widehat{L}_c(X)/\widehat{I}$ where $|\widehat{L}_c(X)/\widehat{I}: L/J|$ is $(p,c,d)$-bounded.

Finally, by Proposition \ref{prop: powerfulpcentralrank} , $\widehat{L}_c(X)/\widehat{I}$ has the $\Omega$EP, that is, $\widehat{L}_c(X)/\widehat{I}\cong \widehat{L}/\Omega_1(\widehat{L})$ for some $p$-central $\Z_p$-Lie algebra. Thus, the result holds.
\end{proof}

Now we can translate this result to $p$-groups of small nilpotency class by the Lazard correspondence.

\begin{Thm} \label{theorem:Groups_structure} Let $p$ be a prime number. Let $G$ be a $d$-generated finite $p$-group of nilpotency class $c<p$. 
Then there exist a powerful $p$-central $p$-group $\widehat{G}$ with $(d,c)$-bounded number of generators and a  normal subgroup  $N$ of $G$ such that
\begin{enumerate}
\item[(a)] $|N|$ is $(p,c,d)$-bounded.
\item[(b)] $G/N$ can be embedded as a subgroup of $\widehat{G}/\Omega_1(\widehat{G})$.
\item[(c)] $|\widehat{G}/\Omega_1(\widehat{G}):G/N|$ is $(p,c,d)$-bounded.
\end{enumerate}
\end{Thm}

\begin{proof}
The claim follows from Theorem \ref{theorem:Liealgebras_structure} and Theorem \ref{thm: Lazard}
\end{proof}

\section{Counting arguments}

In this section, we shall prove some technical results that will be necessary to prove the main result. Actually, the following results can be considered as a generalization of Lemma 4.3 and Theorem 4.4 in \cite{DGG}.

\begin{Lem}\label{lemma:CpGQ}
Let $p$ be an odd prime, let $r,c$ be integral numbers and let $G$ be a $p$-group with $\rk(G)\leq r$. Let
\begin{equation}
 \xymatrix{
1 \ar[r] &C_p \ar[r] &G\ar[r]^{\pi} &Q \ar[r] &1,}
\end{equation}
be an extension of groups. Suppose that $Q$ has a subgroup $A$ of nilpotency class $c<p$. Set $B=\pi^{-1} (A)^{p^2}$. Then $B$ is a powerful $p$-central $p$-group of nilpotency class at most $c$ with the $\Omega$EP and $|G:B|\leq p^{2cr}|Q:A|$.
\end{Lem}

\begin{proof}
Put $C=\pi^{-1} (A)$, $D=C^p$ and let $N$ be the image of $C_p$ in $G$. Since the nilpotency class of $A$ is $c$, we know that the nilpotency class of $C$ is at most $c+1$. Then, we may write
\begin{align*}
\gamma_{p-1}(D)=[\gamma_{p-2}(C^p),C^p]\leq [\gamma_{p-2}(C^p),C]^p\subseteq N^{p}=1,
\end{align*}
where in the inequality we used the fact that $C$ is a regular $p$-group and Lemma 2.13 in \cite{Gustavo2000}.
In particular, $D$ is a $p$-group of nilpotency class at most $c$ and by Proposition \ref{prop:GpOmega}, $B=D^p$ is a powerful $p$-central group with the $\Omega$EP.

We also have $|G:B|=|G:C||C:D||D:B|$ and $|G:C|=|Q:A|$, where $C/D$ and $D/B$ have exponent $p$. Moreover, $C$ has rank at most $r$ and nilpotency class at most $c+1$. We may write 
\[
|C:D|=|C: D\gamma_2(C)|| D\gamma_2(C): D\gamma_3(C)|\dots  |D\gamma_{c}(C):D|.
\]
Note that the quotients $D\gamma_i(C)/D\gamma_{i+1}(C)$ are elementary abelian and therefore 
$$|D\gamma_i(C):D\gamma_{i+1}(C)|\leq p^r.$$
In particular $|C:D|\leq p^{cr}$. A similar argument shows that $|D:B|\leq p^{cr}$. Then, the bound in the statement follows.
\end{proof}

We continue with the main result of this section which is a generalization of Theorem 3.3 in \cite{Carlson}.

\begin{Thm}\label{thm:cohringsfromquotient}
Let $p$ be an odd prime, let $c, r, n, f$ be positive integers and suppose that 
\begin{equation}
\label{eq:cohringsfromquotient}
1 \to H \to G \to Q \to 1
\end{equation}
is an extension of finite $p$-groups with $|H|\leq p^n$, $\rk(G)\leq r$ and $Q$ has a subgroup $A$ of nilpotency class $c<p$ with $|Q:A| \leq f$. Then the ring $H^*(G;\F_p)$ is determined up to a finite number of possibilities (depending on $p$, $n$, $r$ and $f$) by the ring $H^*(Q;\F_p)$.
\end{Thm}

\begin{proof}

We start with the base case $H\cong C_p$. By Lemma \ref{lemma:CpGQ}, there exist a powerful $p$-central subgroup $B$ of $G$ with $\Omega$EP and whose index is bounded in terms of $p$, $r$ and $f$. If $H$ is not contained in $B$ we consider $H\times B$ instead of $B$. In both situations, there exists an element $\eta\in H^2(B,\F_p)$ (resp. $\eta\in H^2(B\times H,\F_p)$) such that $\text{res}^B_H( \eta )$ is non-zero (resp. $\text{res}^{B\times H}_H( \eta )$) (see Theorem \ref{Thm:Weigel}). Then the spectral sequence arising from 
\[
1 \to H \to G \to Q \to 1
\]
stops at most at the page $2|G:B|+1$ (cf. \cite[proof of Lemma 3.2]{Carlson}). Now the theorem holds by \cite[Proposition 3.1]{Carlson}.

For general $H$, we proceed by induction on $|H|$. Suppose that the result holds for all the group extensions of the form
\[
1\to H'\to G\to Q,
\] 
where $|H'|<|H|\leq p^n$, $\rk(G)\leq r$ and with $A\leq Q$ of nilpotency class $c<p$ and $|Q:A|\leq f$. Choose a subgroup $H'\leq H$ with $H' \unlhd G$ and $|H:H'|=p$. The quotients $G'=G/H'$ and $C_p\cong H/H'$ fit in a short exact sequence
\begin{equation}\label{eq:applyinginductionextensiongroups}
1 \to C_p \to G' \overset{\pi}\to Q \to 1
\end{equation}
and we also have the following extension of groups,
\begin{equation}\label{eq:centralextensioninduction}
1\to H'\to G\to G'\to 1.
\end{equation}
Applying Lemma \ref{lemma:CpGQ} to the extension of groups \eqref{eq:applyinginductionextensiongroups}, we know that $G'$ has a powerful $p$-central $p$-subgroup $B'$ with the $\Omega$EP with $|G':B'|\leq p^{2cr}|Q:A|$. Also, by the previous case, we have that the cohomology algebra $H^*(G';\F_p)$ is determined up to a finite number of possibilities (depending on $p$, $n$, $r$ and $f$) by the algebra $H^*(Q;\F_p)$.

Now, we may apply the induction hypothesis to the extension \eqref{eq:centralextensioninduction} since $|H'|<|H|$. Then, the cohomology algebra $H^*(G;\F_p)$ is determined up to a finite number of possibilities (depending on $p$, $n$, $r$ and $f$) by the algebra $H^*(G';\F_p)$. In turn, the result holds.
\end{proof}

\section{Main result and further work}

In this section we prove the main result of this paper and we mention the main obstructions to extend this result.

\begin{Thm}
Let $p$ be a prime number and let $d$ be a non-negative integer. Then the number of possible isomorphism types for the mod $p$-cohomology algebra of a $d$-generated $p$-group of nilpotency class smaller than $p$ is bounded by a function depending on $p$ and $d$.
\end{Thm}

\begin{proof}Let $G$ be a $d$-generated finite $p$-group of nilpotency class smaller than $p$. Notice that 
the elementary abelian quotients $\gamma_i(G)G^p/\gamma_{i+1}(G)G^p$ can be generated by at most $d^i$ commutators. Therefore,
\begin{eqnarray*}
|G:G^p| & = & |\gamma_1(G)G^p: \gamma_2(G)G^p||\gamma_2(G)G^p:\gamma_3(G)G^p|\dots |\gamma_{p-1}(G)G^p:G^p| \\
& \leq & p^{d+d^2+\ldots + d^{p-1}}.
\end{eqnarray*}
Since $G$ is a regular $p$-group (see \cite[Theorem 2.8 (i)]{Gustavo2000}), for any subgroup $H$ of $G$, we have that
$$|H:\Phi (H)|\leq |H:H^p|=|\Omega_1(H)|\leq |\Omega_1(G)|= |G:G^p|\leq p^{d+d^2+\ldots +d^{p-1}}.$$
In particular, $\rk(G)\leq d+d^2+\ldots + d^{p-1}$ (see \cite[Theorem 2.8 (i) and Theorem 2.10 (iv)]{Gustavo2000}).

By Theorem \ref{theorem:Groups_structure}, there exist a powerful $p$-central group $\widehat{G}$  and a normal subgroup $N$ of $G$ such that
\begin{enumerate}[(i)]\label{propertiesmainproof}
\item The number of generators of $\widehat{G}$ is $(d)$-bounded,
\item The order $|N|$ is $(p,d)$-bounded,
\item\label{eq:point3proof}
$G/N$ can be embedded as a subgroup of $\tilde{G}=\widehat{G}/\Omega_1(\widehat{G})$ whose index if $(p,d)$-bounded.
\end{enumerate}
Since $\tilde{G}$ is poweful $p$-central the number of generators of $\tilde{G}$ coincide with the number of generators of $\Omega_1(\tilde{G})$ (compare  
\cite[Theorem 6.5]{JaikinGonzalez} and the facts that $\tilde{G}$ is powerful and $\Omega_1(\tilde{G})$ is abelian).  Furthermore $\tilde{G}$ has the $\Omega$EP. Therefore by Theorem \ref{Thm:Weigel}, $H^*(\tilde{G};\F_p)$ is isomorphic to the graded $\F_p$-algebra
 $$\Lambda(y_1, \dots, y_e) \otimes \F_p[x_1, \dots, x_e],$$
where $e$ denotes the number of generators of $\tilde{G}$ and $|y_i|=1$ and $|x_i|=2$. Recall that the possibilities for $e$ are bounded in terms of $p$ and $d$ ($e$ is at most $d+d^2+\ldots + d^{p-1}$). By Theorem \cite[Theorem 3.5]{Carlson} and point \eqref{eq:point3proof} above, the number of isomorphism types for $H^*(G/N;\F_p)$ is bounded in terms of $p$ and $d$. 

Finally, since the rank of $G$ is $(p,d)$-bounded, the nilpotency class of $G/N$ is smaller than $p$ and $|N|$ is $(p,d)$-bounded, applying Theorem \ref{thm:cohringsfromquotient} to the extension
$$1\to N\to G\to G/N\to 1$$
and to the group $A=G/N$ we conclude that the number of isomorphism types of algebras for $H^*(G;\F_p)$ is bounded in terms of $p$ and $d$. This concludes the proof of the  main result of this paper.
\end{proof}


The main obstruction to extend this result to $p$-groups of arbitrary nilpotency class is the absence of a nice Lie theory for such groups. However, we think that an analogous result to that of Theorem \ref{theorem:Groups_structure} should hold for $p$-groups of  arbitrary nilpotency class $c$. Thus, we propose the following conjecture.

\begin{Conj}
Let $p$ be a prime number and let $c$ and $d$ be non-negative integers. Then the number of possible isomorphism types for the mod-$p$ cohomology ring of a $d$-generated $p$-group of nilpotency class $c$ is bounded by a function depending on $p$, $d$ and $c$.
\end{Conj}

Finite $p$-groups of fixed nilpotency class or fixed coclass $c$ share one common property, namely, that both have rank bounded by some function depending on $p$ and $c$. The following conjecture encompasses  the above conjecture as well as Carlson's conjecture for the cohomology of finite $p$-groups of fixed coclass (see \cite[Section 6]{Carlson}).

\begin{Conj}
Let $p$ be prime number and let $r$ be a non-negative integer. Then the number of possible isomorphism types for the mod-$p$ cohomology ring of a  $p$-group of rank $r$  is bounded by a function depending on $p$ and $r$.
\end{Conj}

One argument in favor of this conjecture is that if $G$ is a $p$-group of rank $r$, then it contains powerful $p$-central subgroup with $\Omega$EP whose index is bounded in terms of $p$ and $r$. The existence of a powerful $p$-central subgroup is classical (see, for example, \cite[\S 11]{khukhro}). The fact that this subgroup can be chosen to have the $\Omega$EP, at least for $p\geq 5$, can be deduced from the Ph.D. thesis \cite{GSPhD} of the last author or from the Habilitation Schrift of T. S. Weigel \cite{WeigelHS}.

\end{document}